\documentclass[12pt]{amsart}
\usepackage{amsmath,amssymb,amsbsy,amsfonts,latexsym,amsopn,amstext,
                                               amsxtra,euscript,amscd,bm}
                    
\usepackage{url}
\usepackage[colorlinks,linkcolor=blue,anchorcolor=blue,citecolor=blue]{hyperref}
\usepackage{cleveref}
\usepackage{color}
\numberwithin{equation}{section}
\usepackage[left=70pt, right=70pt]{geometry}
\usepackage{mathtools}
\usepackage{enumitem}
\usepackage{dsfont}

\begin{document}

\newtheorem{problem}{Problem}
\newtheorem{theorem}{Theorem}[section]
\newtheorem{lemma}[theorem]{Lemma}
\newtheorem{corollary}[theorem]{Corollary}
\newtheorem{conjecture}[theorem]{Conjecture}
\newtheorem{example}[theorem]{Example}
\newtheorem{claim}[theorem]{Claim}
\newtheorem{cor}[theorem]{Corollary}
\newtheorem{prop}[theorem]{Proposition}
\newtheorem{definition}[theorem]{Definition}
\newtheorem{question}[theorem]{Question}
\newtheorem{conj}{Conjecture}
\newtheorem{hypothesis}[theorem]{Hypothesis}
\newtheorem{condition}[theorem]{Condition}
\def\cA{{\mathcal A}}
\def\cB{{\mathcal B}}
\def\cC{{\mathcal C}}
\def\cD{{\mathcal D}}
\def\cE{{\mathcal E}}
\def\cF{{\mathcal F}}
\def\cG{{\mathcal G}}
\def\cH{{\mathcal H}}
\def\cI{{\mathcal I}}
\def\cJ{{\mathcal J}}
\def\cK{{\mathcal K}}
\def\cL{{\mathcal L}}
\def\cM{{\mathcal M}}
\def\cN{{\mathcal N}}
\def\cO{{\mathcal O}}
\def\cP{{\mathcal P}}
\def\cQ{{\mathcal Q}}
\def\cR{{\mathcal R}}
\def\cS{{\mathcal S}}
\def\cT{{\mathcal T}}
\def\cU{{\mathcal U}}
\def\cV{{\mathcal V}}
\def\cW{{\mathcal W}}
\def\cX{{\mathcal X}}
\def\cY{{\mathcal Y}}
\def\cZ{{\mathcal Z}}

\def\A{{\mathbb A}}
\def\B{{\mathbb B}}
\def\C{{\mathbb C}}
\def\D{{\mathbb D}}
\def\E{{\mathbb E}}
\def\F{{\mathbb F}}
\def\G{{\mathbb G}}
\def\I{{\mathbb I}}
\def\J{{\mathbb J}}
\def\K{{\mathbb K}}
\def\L{{\mathbb L}}
\def\M{{\mathbb M}}
\def\N{{\mathbb N}}
\def\O{{\mathbb O}}
\def\P{{\mathbb P}}
\def\Q{{\mathbb Q}}
\def\R{{\mathbb R}}
\def\S{{\mathbb S}}
\def\T{{\mathbb T}}
\def\U{{\mathbb U}}
\def\V{{\mathbb V}}
\def\W{{\mathbb W}}
\def\X{{\mathbb X}}
\def\Y{{\mathbb Y}}
\def\Z{{\mathbb Z}}

\def\e{{\mathbf{e}}}
\def\ep{{\mathbf{e}}_p}
\def\eq{{\mathbf{e}}_q}

\newcommand{\argmin}{\operatornamewithlimits{argmin}}

\DeclarePairedDelimiter\ceil{\lceil}{\rceil}
\DeclarePairedDelimiter\floor{\lfloor}{\rfloor}

\def\scr{\scriptstyle}
\def\({\left(}
\def\){\right)}
\def\[{\left[}
\def\]{\right]}
\def\<{\langle}
\def\>{\rangle}
\def\fl#1{\left\lfloor#1\right\rfloor}
\def\rf#1{\left\lceil#1\right\rceil}
\def\le{\leqslant}
\def\ge{\geqslant}
\def\eps{\varepsilon}
\def\mand{\qquad\mbox{and}\qquad}

\def\sssum{\mathop{\sum\ \sum\ \sum}}
\def\ssum{\mathop{\sum\, \sum}}
\def\ssumw{\mathop{\sum\qquad \sum}}

\def\vec#1{\mathbf{#1}}
\def\inv#1{\overline{#1}}
\def\num#1{\mathrm{num}(#1)}
\def\dist{\mathrm{dist}}

\def\fA{{\mathfrak A}}
\def\fB{{\mathfrak B}}
\def\fC{{\mathfrak C}}
\def\fU{{\mathfrak U}}
\def\fM{{\mathfrak M}}
\def\fm{{\mathfrak m}}
\def\fV{{\mathfrak V}}

\newcommand{\bflambda}{{\boldsymbol{\lambda}}}
\newcommand{\bfxi}{{\boldsymbol{\xi}}}
\newcommand{\bfrho}{{\boldsymbol{\rho}}}
\newcommand{\bfnu}{{\boldsymbol{\nu}}}
\newcommand{\psum}{\sideset{}{^*}\sum}

\def\GL{\mathrm{GL}}
\def\SL{\mathrm{SL}}

\def\Hba{\overline{\cH}_{a,m}}
\def\Hta{\widetilde{\cH}_{a,m}}
\def\Hb1{\overline{\cH}_{m}}
\def\Ht1{\widetilde{\cH}_{m}}

\def\flp#1{{\left\langle#1\right\rangle}_p}
\def\flm#1{{\left\langle#1\right\rangle}_m}
\def\dmod#1#2{\left\|#1\right\|_{#2}}
\def\dmodq#1{\left\|#1\right\|_q}

\newcommand*\diff{\mathop{}\!\mathrm{d}}
\newcommand{\funccol}{\colon \thinspace}
\setlength\parindent{0pt}

\def\Zm{\Z/m\Z}

\def\Err{{\mathbf{E}}}

\newcommand{\comm}[1]{\marginpar{%
\vskip-\baselineskip 
\raggedright\footnotesize
\itshape\hrule\smallskip#1\par\smallskip\hrule}}

\def\xxx{\vskip5pt\hrule\vskip5pt}



\title[Primes in short arithmetic progressions]{Refinements for primes in short arithmetic progressions}
\author{Michael Harm}
\address{School of Mathematics and Statistics, UNSW, 
Sydney, NSW 2052, Australia}
\email{m.harm@unsw.edu.au}

\begin{abstract}
    Given a zero-free region and an averaged zero-density estimate over all Dirichlet $L$-functions modulo $q\in\N$, we refine the error terms of the prime number theorem in all and almost all short arithmetic progressions. For example, if we assume the Generalized Density Hypothesis, then for any arithmetic progression modulo $q\leq \log^{\ell} x$ with $\ell>0$ and any $\varepsilon>0$, the prime number theorem holds in all intervals $(x-\sqrt{x}\exp(\log^{\frac{2}{3}+\varepsilon} x),x]$ and almost all intervals $(x-\exp(\log^{\frac{2}{3}+\varepsilon} x),x]$ as $x\rightarrow\infty$. This refines the classic intervals $(x-x^{1/2+\varepsilon},x]$ and $(x-x^\varepsilon,x]$ for any $\varepsilon>0$.
\end{abstract}
\maketitle
\thispagestyle{empty}

\section{Introduction}
Let $p$ always denote a prime, and let $a,q\in\N$ with $(a,q)=1$. We denote the error term of the (unweighted) prime counting function in a short arithmetic progression $a$ modulo $q$ by
\begin{equation*}
    \Delta_\pi(x,y,q,a):=\sum_{\substack{x-y< p\leq x\\p\equiv a\thinspace (q)}}1 -\frac{\text{li}(x,y)}{\varphi(q)},
\end{equation*}
where $\varphi$ denotes the Euler totient function and $\text{li}(x,y):=\int_{x-y}^x\frac{\diff t}{\log t}$ denotes the logarithmic integral. We say that the Prime Number Theorem (PNT) holds for an arithmetic progression $a$ modulo $q$ in the interval $(x-y,x]$, with $y=y(x)$, if 
\begin{equation}\label{eq: PNT holds for all}
    \Delta_\pi(x,y,q,a)= o\bigg(\frac{y}{\varphi(q)\log x}\bigg) \quad \text{as} \quad x\rightarrow\infty.
\end{equation}
The connection between prime counting functions and zero-free regions of Dirichlet $L$-functions is well known via Perron's formula. For a Dirichlet character $\chi$ modulo $q$, let $\cN(\sigma, T,\chi)$ denote the number of zeroes $\rho=\beta+i\gamma$ of the Dirichlet $L$-function $L(s,\chi)$ with $\beta> \sigma$ and $|\gamma|\leq T$.

\begin{condition}\label{hyp: zero-free region}
    Let $q\leq Q$ and let $T_0>0$ (possibly) depend on $Q$. There exists a continuous, non-increasing function $0<\eta(T)\leq 1/2$ in $T\geq T_0$, such that
    \begin{equation*}
        \cN(1-\eta(T),T,\chi)=0,
    \end{equation*}
    for all Dirichlet characters $\chi$ modulo $q$, except possibly for the quadratic character modulo $q$, in which case there may be at most one real "exceptional" zero $\beta_0<1$ in this region. 
\end{condition}
We also denote $\eta_0:=\lim_{T\rightarrow \infty}\eta(T)$ and
    \begin{equation*}
        \omega(x):=\log x \thinspace\eta(\argmin_{T_0\leq T\leq x}(\log x \thinspace\eta(T)+\log T)),
    \end{equation*}
where $\argmin_S(f)$ denotes the element of the set $S$ which minimizes the function $f$. Note that we slightly altered the conventional definition of $\omega$ to include constant functions $\eta(T)=\eta_0$. The best currently known zero-free region is given by
\begin{equation}\label{eq: Vinogradov-Korobov region}
    \eta(T)= \frac{c}{(\log T)^{2/3} \thinspace(\log\log T)^{1/3}}
\end{equation}
for $q\leq \exp(\eta^{-1}(T))$ and some constant $c>0$. For the principal character this has been independently proven by Vinogradov \cite{Vin58} and Korobov \cite{Kor58}. For an (explicit) generalization to $q\geq 3$ see e.g. \mbox{Khale \cite[Theorem 1.1]{Kha23}}. In this case we have $\omega(x)\asymp (\log x)^{3/5}(\log\log x)^{-1/5}$ \mbox{(see \cite[Lemma A.3]{JY23}).} The Generalized Riemann Hypothesis (GRH) conjectures that we may take $\eta_0=1/2$ for all $q\in\N$ (and the absence of the exceptional zero), however, this seems to be beyond the current scope of mathematics. Ingham \cite{Ing32} showed that, assuming Condition \ref{hyp: zero-free region}, we have 
\begin{equation*}
    \Delta_\pi(x,x,1,1)\ll x\exp(-(1-\varepsilon) \omega(x))
\end{equation*}
for any $\varepsilon>0$, implying that \eqref{eq: PNT holds for all} holds for $q=1$ and $y\in[x\exp(-(1-\varepsilon) \omega(x)), x]$. \\

\begin{condition}\label{hyp: zero-density}
    Let $q\in\N$ and let $T_0>0$ (possibly) depend on $q$. There exist a constant $A\geq 2$ and a function $1\ll g(q,T)\ll (qT)^{o(1)}$, such that for $T\geq T_0$ and $1/2\leq\sigma\leq 1$, we have
\begin{equation*}                                                                                  
        \sum_{\chi\thinspace(q)}\cN(\sigma,T,\chi)\ll (qT)^{A(1-\sigma)}g(q,T),
    \end{equation*}
    where the sum runs over all Dirichlet characters $\chi$ modulo $q$.
\end{condition}
The motivation for this paper is in part due to recent advancements regarding zero-density estimates and to give a direct correspondence between those and the PNT in short arithmetic progressions. Until recently the best zero-density estimate for Dirichlet $L$-functions could be attributed to Huxley with $A=\frac{12}{5}$ \cite[Equation 1.1]{Hux75}. However, a recent result by Guth-Maynard \cite{GM24} gives $A=\frac{30}{13}$ for $q=1$. Chen \cite{Che25} was able to improve Huxley's result using the methods of Guth-Maynard, per Chen's result we may take $A=\frac{7}{3}$ for all $q\in\N$. Barring any major roadblocks, it is feasible to expect Chen's result to be further improved to $A=\frac{30}{13}$ for any $q\in\N$. Note that if Condition \ref{hyp: zero-free region} is true for a constant $\eta_0>0$, then \mbox{Theorem \ref{thm: vertical zeroes}} below implies a zero-density estimate with $A=\frac{1}{\eta_0}$ and $g(q,T)=\log qT$. Therefore the Generalized Density Hypothesis, which is the case $A=2$, is considered a weaker version of GRH.\\

Using zero-density estimates, Hoheisel \cite{Hoh30} was the first to make significant progress towards the PNT in short intervals. He showed that \eqref{eq: PNT holds for all} holds for $q=1$ and $y=x^{1-\frac{1}{33000}+\varepsilon}$ for any $\varepsilon>0$. Ingham \cite{Ing37} improved and generalized Hoheisel's approach. He showed that Condition \ref{hyp: zero-density} implies that \eqref{eq: PNT holds for all} holds for $q=1$ and $y=x^{1-\frac{1}{A}+\varepsilon}$, for any $0<\varepsilon\leq \frac{1}{A}$. It is well-known that this can be generalized to all sufficiently small $q$. Recently \mbox{Starichkova \cite{Sta25}} refined Ingham's result, such that \eqref{eq: PNT holds for all} holds for $q=1$ and $y=x^{1-\frac{1}{A}}\exp( \log^{\frac{2}{3}+\varepsilon} x)$ for any $\varepsilon>0$, if we assume Condition \ref{hyp: zero-free region}. The first set of results in this paper aim to use Ingham's method to generalize Starichkova's refined result to arithmetic progressions modulo $q$. Furthermore, we provide an improved error term, this can be useful for e.g. circle method proofs, where we often need some logarithmic powers as leeway.\\

Unfortunately, Ingham's method will not allow us to prove the PNT for intervals of size $\cO(\sqrt{x})$ due to the symmetry of the zeroes of Dirichlet $L$-functions. For most practical applications, however, it often suffices if the PNT in arithmetic progressions modulo $q$ holds for "almost all" short intervals averaged over $q$. That is, for some $q\in\N$ and $h=h(X)$, we have
\begin{equation}\label{eq: PNT for almost all}
    \psum_{a\thinspace (q)}\int_X^{2X}|\Delta_\pi(u,h,q,a)|^2\diff u = o\bigg(\frac{h^2X}{\varphi(q)\log^2 X}\bigg),
\end{equation}
as $X\rightarrow\infty$, where the sum runs through a set of representatives of the multiplicative group $\Z/q\Z^*$. The first result of this type is due to Selberg \cite{Sel43}, who proved for $0<\theta\leq X^{-1/4}$ and under the assumption of the Riemann Hypothesis an upper bound for the integral
\begin{equation}\label{eq: Selberg}
    \int_1^{X}|\Delta_\pi(u,\theta u,1,1)|^2 u^{-2} \diff u\ll \theta ,
\end{equation}
which implies that \eqref{eq: PNT for almost all} holds for $q=1$ and $h\in[\log^{2+\varepsilon}X,X]$ for any $\varepsilon>0$. He also gave an unconditional result derived from a zero-free region by Chudakov \mbox{$\eta(T)\asymp (\log\log T)^2/\log T$} and Ingham's zero-density estimate $A=77/29$, thereby implying that \eqref{eq: PNT for almost all} holds for \mbox{$q=1$} and $h\in[X^{1-\frac{2}{A}+\varepsilon},X]$. Prachar \cite{Pra74} generalized Selberg's result \eqref{eq: Selberg} to averages over arithmetic progressions, i.e. that \eqref{eq: PNT for almost all} holds for $h\in[q\log^{2+\varepsilon}qX,X]$, for any $\varepsilon>0$.\\ \mbox{Saffari and Vaughan \cite{SV77}} formalized Selberg's implied result in terms of arbitrary zero-density estimates. By assuming the pair correlation conjecture (and GRH), Goldston-Yildirim \mbox{\cite{GY98}} were able to make various improvements to Prachar's bound for short segments of arithmetic progressions. There also exist similar results of the form
\begin{equation}\label{eq: average unsquared PNT}
    \sum_{q\leq Q}\max_{a\thinspace (q)}\int_X^{2X} |\Delta_\pi(u,h,q,a)|\diff u=o\bigg(\frac{hX}{\log X}\bigg).
\end{equation}
The strongest such result is by Koukoulopoulos \cite{Kou15}, who, assuming Condition \ref{hyp: zero-density} holds averaged over $q\leq Q$, showed that we have \eqref{eq: average unsquared PNT} for $1\leq h\leq X$ and $1\leq Q^2\leq h /X^{1-\frac{2}{A}+\varepsilon}$ for any $\varepsilon\in (0,1/3]$.
For the second set of results in this paper we refine the results of Saffari-Vaughan \cite{SV77} and generalize them to hold for arithmetic progressions. We state these results conditionally on the reader's choice of zero-free regions and zero-density estimates. This result is again particularly useful to bound error terms on the major arcs for the Hardy-Littlewood circle method. Since for $h\geq 1$, \mbox{Gallagher's inequality \cite[Lemma 1]{Gal70}} and Lemma \ref{lemma: int psi to chi} give
\begin{equation*}\label{eq: gallagher}
    \int_{-h^{-1}}^{h^{-1}}\bigg| \sum_{X<p\leq 2X}\e\bigg(p(\alpha+\frac{a}{q})\bigg) -\frac{\mu(q)}{\varphi(q)}\int_X^{2X}\frac{\e(t\alpha)}{\log t}\diff t \bigg|^2\diff \alpha\ll \frac{q}{h^2}\int_X^{2X} |\Delta_\pi(u,\frac{h}{2},q,a)|^2\diff u +X.
\end{equation*}

\section{Main results}

Note that while the results of this paper are formulated in terms of the error terms $\Delta_\pi(x,y,q,a)$ they can easily be modified into the error terms of other twisted prime counting functions by \mbox{Lemmas \ref{lemma: psi to vartheta}--\ref{lemma: int psi to chi}.} This is particularly relevant for applications to e.g. the circle method, as the error terms $\Delta_\pi(x,y,a/q)$ (see \ref{it: additive}) typically have an extra factor of $\sqrt{q}$ from the associated Gauss sums.

\subsection{Primes in all short arithmetic progressions}
Obviously Condition \ref{hyp: zero-free region} is far stronger than Condition \ref{hyp: zero-density}, in fact, $\eta_0>0$ induces a zero density estimate with $A=\frac{1}{\eta_0}$. We therefore may always assume $A\leq \frac{1}{\eta_0}$. Using zero-density estimates may also lead to worse error terms as $y\rightarrow x$ and may incur additional small logarithmic factors, which is relevant if we e.g. assume GRH. It is therefore useful to have Ingham's result \cite{Ing32} as a baseline.
\begin{theorem}\label{thm: Ingham}
    Assume Condition \ref{hyp: zero-free region}. Let $q\leq Q$ and $y\in [qx \exp(-\omega(x))\log^{2+\varepsilon} qx, x]$ for any $\varepsilon>0$. We have
    \begin{equation*}
         \Delta_\pi(x,y,q,a)\ll \frac{\cB(x,q)}{\log x} + x\log qx \exp(-\omega(x)).
    \end{equation*}
    where $\cB(x,q)=\frac{x^{\beta_0}}{\varphi(q)}$ if the $L$-function associated to the quadratic Dirichlet character modulo $q$ has an exceptional zero $\beta_0$ and $\cB(x,q)=0$ otherwise. We may omit the term $\cB$ if we additionally restrict $q\leq (\frac{\log x}{\omega(x)})^\ell$ for any $\ell>0$.
\end{theorem}
Theorem \ref{thm: Ingham} together with the Vinogradov-Korobov zero free region (see \cite{Vin58}, \cite{Kor58} \& \cite{Kha23}) imply the following unconditional Corollary due to \cite[Lemma A.3]{JY23}.
\begin{corollary}
    Let $q\leq \log^\ell x$ for any $\ell>0$ and let $0<\varepsilon_1<\varepsilon_2<1$. There exists a constant $c>0$ such that for any 
    \begin{equation*}
        y\in\bigg[x\exp\bigg(-\varepsilon_1\frac{c\log^\frac{3}{5}x}{(\log\log x)^\frac{1}{5}}\bigg),x\bigg],
    \end{equation*}
    we have
    \begin{equation*}
        \Delta_\pi(x,y,q,a)\ll x\exp\bigg(-\varepsilon_2\frac{c\log^\frac{3}{5}x}{(\log\log x)^\frac{1}{5}}\bigg). 
    \end{equation*}
\end{corollary}
If we assume GRH instead, Theorem \ref{thm: Ingham} implies the following Corollary.
\begin{corollary}
    Assume GRH. Let $y\in[q\sqrt{x}\log^{2+\varepsilon}qx,x]$ for any $\varepsilon>0$. We have 
    \begin{equation*}
        \Delta_\pi(x,y,q,a)\ll \sqrt{x}\log qx.
    \end{equation*}
    
\end{corollary}

By invoking zero-density estimates we can significantly reduce the size of the arithmetic progressions in which the PNT holds.

\begin{theorem}\label{thm: primes in all intervals}
    Assume Conditions \ref{hyp: zero-free region} \& \ref{hyp: zero-density} with $\eta_0<\frac{1}{A}$. Let $q\leq Q$ and 
    \begin{equation*}
        y\in\bigg[q^2x^{1-\frac{1}{A}}\exp\bigg(\frac{\eta^{-1}(x^{1/A})}{A}\log(g(q,x))\bigg)\log^{1+\tau+\varepsilon} qx,x\bigg],
    \end{equation*} for any $\varepsilon>0$ and where $\tau=\tau(x,A,\eta)=\frac{1}{1+A\eta(x^{1/A})}$. We have
    \begin{equation}\label{eq: error term main}
        \Delta_\pi(x,y,q,a) \ll \frac{y\cB(x,q)}{x\log x}+ y \bigg(\frac{x^{1-\frac{1}{A}}}{y}\bigg)^{1-\tau}q^{1-2\tau}g^\tau(q,x)\log^{1-\tau} qx  .
    \end{equation}
    We may omit the term $\cB$ if we additionally restrict $q\leq \log^\ell x$ for any $\ell>0$.
\end{theorem}
While zero-density estimates are great for reducing the size of $y$, they don't contribute much saving for the error term \eqref{eq: error term main}. Clearly, the saving in the error term \eqref{eq: error term main} is mostly dependent on $(\frac{x^{1-\frac{1}{A}}}{y})^{1-\tau}$, for $\frac{1}{2}<\tau<1$. This indicates that better savings are mostly attributed to \mbox{zero-free regions} with $A\eta(x)\rightarrow1$. Note that for $\tau=\frac{1}{2}$ we may refer to Theorem \ref{thm: Ingham}. Theorem \ref{thm: primes in all intervals} induces the following Corollaries.
\begin{corollary}\label{cor: all korobov 1}
    Let $0<\varepsilon_1<\varepsilon_2\leq \frac{1}{A}$ and $q\leq \log^{\ell} x$ for any $\ell>0$. Let $c>0$ denote a constant satisfying the Vinogradov-Korobov zero-free region \eqref{eq: Vinogradov-Korobov region}. Assume Condition \ref{hyp: zero-density} with 
    \begin{equation*}
        g(q,T)\ll \exp\bigg(\varepsilon_3\thinspace A^2\thinspace c\bigg(\frac{\log T^\frac{1}{A}}{\log\log T^\frac{1}{A}}\bigg)^\frac{1}{3}\bigg) 
    \end{equation*} 
    for any $0<\varepsilon_3<\varepsilon_2-\varepsilon_1$. For any $y\in [x^{1-\frac{1}{A}+\varepsilon_2},x]$ we have
    \begin{equation*}
        \Delta_\pi(x,y,q,a)\ll y\exp\bigg(-\varepsilon_1\thinspace A^2\thinspace c\bigg(\frac{\log x^\frac{1}{A}}{\log\log x^\frac{1}{A}}\bigg)^\frac{1}{3}\bigg).
    \end{equation*}
\end{corollary}
We may further decrease the size of $y$ by sacrificing some saving in the error term.
\begin{corollary}\label{cor: all korobov 2}
    Let $0<\varepsilon_1<\varepsilon_2< \frac{1}{3}$ and $q\leq \log^{\ell} x$ for any $\ell>0$. Assume Condition \ref{hyp: zero-density} with $g(q,T)\ll \exp(\log^{\varepsilon_3}T)$ for any $0<\varepsilon_3<\varepsilon_1$. For any
    \begin{equation*}
        y\in [x^{1-\frac{1}{A}}\exp(\log^{\frac{2}{3}+\varepsilon_2} x),x]
    \end{equation*}
    we have
    \begin{equation*}
        \Delta_\pi(x,y,q,a)\ll y\exp(-\log^{\varepsilon_1}x).
    \end{equation*}
\end{corollary}
Corollaries \ref{cor: all korobov 1} \& \ref{cor: all korobov 2} are due to the (generalized) Vinogradov-Korobov zero-free region (see \cite{Vin58}, \cite{Kor58}, \cite{Kha23}) with the corresponding constant $c>0$. Note that we may take $A=\frac{7}{3}$ unconditionally due to Chen \cite{Che25}. We may also take $A=\frac{30}{13}$ for $q=1$ due to Guth-Maynard \cite{GM24} while we expect a further generalization of their result. 

\begin{corollary}\label{cor: main constant}
    Assume Conditions \ref{hyp: zero-free region} \& \ref{hyp: zero-density} with $\frac{1}{A}>\eta(T)=\eta_0>0$ and $g(q,T)=\log^B qT$ for some $B\geq 0$. We also assume the absence of exceptional zeroes. Let $C=2+\frac{1+B}{A\eta_0}+\varepsilon$ for any $\varepsilon>0$ and let
    \begin{equation*}
        y\in [q^2 x^{1-\frac{1}{A}}\log^C qx,x].
    \end{equation*}
    We have
    \begin{equation*}
        \Delta_\pi(x,y,q,a)\ll \frac{y}{q\log x}(\log qx)^{-\varepsilon \frac{ A\eta_0}{1+A\eta_0}}.
    \end{equation*}
 
\end{corollary}
The Conditions for Corollary \ref{cor: main constant} describe a weak case of GRH dominated by a zero-density estimate, such that $\frac{1}{2}\leq 1-\frac{1}{A}<1-\eta_0$. Note that Corollary \ref{cor: main constant} could be improved to $C>\max(1, \frac{1+B}{A\eta_0})$, if we apply a result by Cully-Hugill and Johnston \cite{CHJ25}. They showed that the error term in the \mbox{Riemann-von Mangoldt} truncated explicit formula can be improved on average. However, they have only proven this result for the case $q=1$.

\subsection{Primes in almost all short arithmetic progressions}
In many practical cases it suffices to show that the PNT in short arithmetic progressions holds for "almost all" intervals. In fact, we may significantly decrease the size $h$, and show that the density of intervals $(u-h,u]$ for $X\leq u\leq 2X$ where the PNT in arithmetic progressions holds is $1$ as $X\rightarrow\infty$. Our first result for this case again serves as a baseline in case the zero-density estimate induced by our zero-free region is the best available.
\begin{theorem}\label{thm: Ingham 2}
    Assume Condition \ref{hyp: zero-free region}. Let $q\leq Q$ and $h\in[q X^{1-2\eta(X)}\log^{2+\varepsilon} qX,X]$ for any $\varepsilon>0$. We have
    \begin{equation*}
        \psum_{a\thinspace(q)}\int_{X}^{2X}| \Delta_\pi(u,h,q,a)|^2 \diff u\ll \frac{h^2\cB(X^2,q)}{X\log^2 qX} + h X^{2-2\eta(X)} .
    \end{equation*}
    We may omit the term $\cB$ if we additionally restrict $q\leq \eta^{-\ell}(X)$ for any $\ell>0$.
\end{theorem}
Theorem \ref{thm: Ingham 2} together with the Vinogradov-Korobov zero free region (see \cite{Vin58}, \cite{Kor58} \& \cite{Kha23}) imply the following unconditional Corollary due to \cite[Lemma A.3]{JY23}.
\begin{corollary}
    Let $q\leq \log^\ell x$ for any $\ell>0$ and let $0<\varepsilon_1<\varepsilon_2<2$. There exists a constant $c>0$ such that for any 
    \begin{equation*}
        h\in\bigg[x\exp\bigg(-\varepsilon_1 \frac{c\log^\frac{3}{5}x}{(\log\log x)^\frac{1}{5}}\bigg),x\bigg],
    \end{equation*}
    we have
    \begin{equation*}
        \psum_{a\thinspace(q)}\int_{X}^{2X}| \Delta_\pi(u,h,q,a)|^2 \diff u\ll h X^2 \exp\bigg(-\varepsilon_2 \frac{c\log^\frac{3}{5}x}{(\log\log x)^\frac{1}{5}}\bigg) .
    \end{equation*}
    For $n\in[X,2X]$ and any co-prime arithmetic progression $a\thinspace(q)$, there are at most
    \begin{equation*}
        \mathcal{O}\bigg(\frac{X^2}{h}\exp\bigg(-\varepsilon_2 \frac{c\log^\frac{3}{5}x}{(\log\log x)^\frac{1}{5}}\bigg)  \bigg)
    \end{equation*} intervals $(n-h,n]$ where the PNT does not hold as $X\rightarrow\infty$. 
\end{corollary}
If we assume GRH instead, Theorem \ref{thm: Ingham 2} implies the following Corollary (which was first shown by {Prachar \cite{Pra74}}).
\begin{corollary}
    Assume GRH. Let $h\in[q \log^{2+\varepsilon} qX,X]$ for any $\varepsilon>0$. We have
    \begin{equation*}
        \psum_{a\thinspace(q)}\int_{X}^{2X}| \Delta_\pi(u,h,q,a)|^2 \diff u\ll h X .
    \end{equation*}
    For $n\in[X,2X]$ and any co-prime arithmetic progression $a\thinspace(q)$, there are at most
    \begin{equation*}
        \mathcal{O}\bigg(\frac{q^2X}{h} \log^2 X\bigg)
    \end{equation*} intervals $(n-h,n]$ where the PNT does not hold as $X\rightarrow\infty$. 

\end{corollary}

We can significantly improve on Theorem \ref{thm: Ingham 2} if we invoke zero-density estimates.

\begin{theorem}\label{thm: primes in almost all intervals}
    Assume Conditions \ref{hyp: zero-free region} \& \ref{hyp: zero-density} with $\eta_0<\frac{1}{A}$. Let $q\leq Q$ and
    \begin{equation*}
        h\in \bigg[qX^{1-\frac{2}{A}}\exp\bigg(\frac{\eta^{-1}(X^\frac{2}{A})}{A}\log(g(q,X)\log^{2+\varepsilon} qX)\bigg),X\bigg]
    \end{equation*}
    for any $\varepsilon>0$. We have 
    \begin{equation*}
    \begin{split}
        \psum_{a\thinspace(q)}\int_X^{2X} |\Delta_\pi(u,h,q,a)|^2\diff u
        \ll& \frac{h^2\cB(X^2,q)}{X\log^2 X}+\frac{h^2X}{\varphi(q)} \bigg( \frac{qX^{1-\frac{2}{A}}}{h}\bigg)^{A\eta(X^\frac{2}{A})}g(q,X)  .
    \end{split}
    \end{equation*}
    We may omit the term $\cB$ if we additionally restrict $q\leq \eta^{\ell}(X)$ for any $\ell>0$.
\end{theorem}

Theorem \ref{thm: primes in almost all intervals} induces the following Corollaries.
\begin{corollary}\label{cor: almost all korobov 1}
    Let $0<\varepsilon_1<\varepsilon_2\leq \frac{2}{A}$ and let $q\leq \log^\ell X$ for any $\ell>0$. Let $c>0$ denote a constant satisfying the Vinogradov-Korobov zero-free region \eqref{eq: Vinogradov-Korobov region}. Assume Condition \ref{hyp: zero-density} with
    \begin{equation*}
        g(q,T)\ll  \exp\bigg(\frac{\varepsilon_3 A^2 c}{2}\bigg(\frac{\log X^\frac{2}{A}}{\log\log X^\frac{2}{A}}\bigg)^\frac{1}{3} \bigg),
    \end{equation*}
    for any $0<\varepsilon_3<\varepsilon_2-\varepsilon_1$. For any $h\in[X^{1-\frac{2}{A}+\varepsilon_2},X]$ we have
    \begin{equation*}
        \psum_{a\thinspace(q)}\int_{X}^{2X}| \Delta_\pi(u,h,q,a)|^2 \diff u\ll h^2X \exp\bigg(-\frac{\varepsilon_1 A^2 c}{2}\bigg(\frac{\log X^\frac{2}{A}}{\log\log X^\frac{2}{A}}\bigg)^\frac{1}{3} \bigg) .
    \end{equation*}
    For $n\in[X,2X]$ and any co-prime arithmetic progression $a\thinspace(q)$, there are at most
    \begin{equation*}
        \mathcal{O}\bigg(X \exp\bigg(-\frac{\varepsilon_1 A^2 c}{2}\bigg(\frac{\log X^\frac{2}{A}}{\log\log X^\frac{2}{A}}\bigg)^\frac{1}{3} \bigg)\bigg)
    \end{equation*} intervals $(n-h,n]$ where the PNT does not hold as $X\rightarrow\infty$. 
\end{corollary}
\begin{corollary}\label{cor: almost all korobov 2}
    Let $0<\varepsilon_1<\varepsilon_2< \frac{1}{3}$ and let $q\leq \log^\ell X$ for any $\ell>0$. Assume Condition \ref{hyp: zero-density} with
    \begin{equation*}
        g(q,T)\ll  \exp(\log^{\varepsilon_3}X),
    \end{equation*}
    for any $0<\varepsilon_3<\varepsilon_1$. For any $h\in[X^{1-\frac{2}{A}}\exp(\log^{\frac{2}{3}+\varepsilon_2}X),X]$ we have
    \begin{equation*}
        \psum_{a\thinspace(q)}\int_{X}^{2X}| \Delta_\pi(u,h,q,a)|^2 \diff u\ll h^2X \exp(-\log^{\varepsilon_1}X ) .
    \end{equation*}
    For $n\in[X,2X]$ and any co-prime arithmetic progression $a\thinspace(q)$, there are at most
    \begin{equation*}
        \mathcal{O}(X \exp(-\log^{\varepsilon_1}X ) )
    \end{equation*} intervals $(n-h,n]$ where the PNT does not hold as $X\rightarrow\infty$. 
\end{corollary}

Corollaries \ref{cor: all korobov 1} \& \ref{cor: all korobov 2} are again due to the (generalized) Vinogradov-Korobov zero-free region (see \cite{Vin58}, \cite{Kor58}, \cite{Kha23}) with the corresponding constant $c>0$. We may again take $A=\frac{7}{3}$ due to Chen \cite{Che25}. We may also take $A=\frac{30}{13}$ for $q=1$ due to Guth-Maynard \cite{GM24} while we expect a further generalization of their result.

\begin{corollary}\label{cor: main constant 2}
    Assume Conditions \ref{hyp: zero-free region} \& \ref{hyp: zero-density} with $\frac{1}{A}>\eta(T)=\eta_0>0$ and $g(q,T)=\log^B qT$ for some $B\geq 0$. We also assume the absence of exceptional zeroes. Let $C=\frac{B+2}{A\eta_0}+\varepsilon$ for any $\varepsilon>0$ and let
    \begin{equation*}
        h\in[qX^{1-\frac{2}{A}}\log^C X,X].
    \end{equation*}
    We have
    \begin{equation*}
        \psum_{a\thinspace(q)}\int_{X}^{2X}| \Delta_\pi(u,h,q,a)|^2 \diff u\ll \frac{h^2X}{q(\log qX)^{2+\varepsilon A\eta_0}}.
    \end{equation*}
    For $n\in[X,2X]$ and any co-prime arithmetic progression $a\thinspace(q)$, there are at most
    \begin{equation*}
        \mathcal{O}\bigg(\frac{qX}{\log^{\varepsilon A\eta_0}qX}\bigg)
    \end{equation*} intervals $(n-h,n]$ where the PNT does not hold as $X\rightarrow\infty$. 
\end{corollary}
The Conditions for Corollary \ref{cor: main constant 2} describe a weak case of GRH dominated by a zero-density estimate, such that $\frac{1}{2}\leq 1-\frac{1}{A}<1-\eta_0$.

\section{Preliminaries}
\subsection{Twisted prime counting functions}\label{sec: Twisted Chebyshev functions}
Let $q\in \N$, let $(a,q)=1$ and let $\chi$ be a Dirichlet character modulo $q$. Also let $p$ always denote a prime, let $\Lambda$ denote the von Mangoldt function, let $\mu$ denote the M\"obius function and let $\varphi$ denote the Euler totient function. We denote the error terms of the prime counting functions in the interval $(x-y,x]$ and
\begin{enumerate}[label=(\roman*)]
    \item modified by a Dirichlet character
\begin{equation*}\label{def: Chebyshev chi}
    \begin{split}
        \Delta_\pi(x,y,\chi):=&\sum_{x-y< p\leq x}\chi(p) -\mathds{1}_{\chi_0}(\chi) \text{li}(x,y),\\
        \Delta_\vartheta(x,y,\chi):=&\sum_{x-y< p\leq x}\chi(p)\log p -\mathds{1}_{\chi_0}(\chi) y,\\
        \Delta_\psi(x,y,\chi):=&\sum_{x-y<n\leq x}\chi(n)\Lambda(n)-\mathds{1}_{\chi_0}(\chi) y,\\
    \end{split}
\end{equation*}
where $\mathds{1}_{\chi_0}$ is the indicator function for the principal character $\chi_0$ modulo $q$.
    \item in an arithmetic progression
\begin{equation*}\label{def: Chebyshev arith prog}
\begin{split}
    \Delta_\pi(x,y,q,a):=&\sum_{\substack{x-y<p\leq x\\p\equiv a\thinspace (q)}}1-\frac{\text{li}(x,y)}{\varphi(q)},\\
    \Delta_\vartheta(x,y,q,a):=&\sum_{\substack{x-y<p\leq x\\p\equiv a\thinspace (q)}}\log p-\frac{y}{\varphi(q)},\\
    \Delta_\psi(x,y,q,a):=&\sum_{\substack{x-y<n\leq x\\n\equiv a\thinspace (q)}}\Lambda(n)-\frac{y}{\varphi(q)}.
\end{split}
\end{equation*}
\item \label{it: additive} or modified by an additive character
\begin{equation*}\label{def: Chebyshev additive}
\begin{split}
    \Delta_\pi(x,y,a/q):=&\sum_{x-y<p\leq x}\e(pa/q)-\frac{\mu(q)}{\varphi(q)}\text{li}(x,y),\\
    \Delta_\vartheta(x,y,a/q):=&\sum_{x-y<p\leq x}\e(pa/q)\log p-\frac{\mu(q)}{\varphi(q)}y,\\
    \Delta_\psi(x,y,a/q):=&\sum_{x-y<n\leq x}\e(na/q)\Lambda(n)-\frac{\mu(q)}{\varphi(q)}y,
\end{split}
\end{equation*}
where $\e(\alpha):=\exp(2\pi i\alpha)$.
\end{enumerate}
The disambiguation between \ref{def: Chebyshev chi} and \ref{def: Chebyshev additive} may be discerned by the third input, which are either a Dirichlet character or a real number. Notably, by taking $q=1$, we obtain the identities
\begin{equation*}
\begin{split}
    \Delta_\pi(x,y,\chi_0)=\Delta_\pi(x,y,1,1)=\Delta_\pi(x,y,1),\\
    \Delta_\vartheta(x,y,\chi_0)=\Delta_\vartheta(x,y,1,1)=\Delta_\vartheta(x,y,1),\\
    \Delta_\psi(x,y,\chi_0)=\Delta_\psi(x,y,1,1)=\Delta_\psi(x,y,1),
\end{split}
\end{equation*}
which are equivalent to the standard (non-twisted) prime counting functions in short intervals. 

All these twisted prime counting functions are closely related. In fact, they can all be written as functions of  $\Delta_\psi(x,y,\chi)$ (thereby incurring small error terms). The following Lemma establishes a connection between the cases $\vartheta$ and $\psi$.
\begin{lemma}\label{lemma: psi to vartheta}
    Let $(a,q)=1$, let $\chi$ be a Dirichlet character modulo $q$ and let $0< y\leq x$. We have 
    \begin{equation*}
    \begin{split}
        \Delta_\vartheta(x,y,\chi)=&\Delta_\psi(x,y,\chi) +\cO(\sqrt{y}\log x),\\
        \Delta_\vartheta(x,y,q,a)=&\Delta_\psi(x,y,q,a) +\cO(\sqrt{y}\log x),\\
        \Delta_\vartheta(x,y,a/q)=&\Delta_\psi(x,y,a/q) +\cO(\sqrt{y}\log x).\\
    \end{split}
    \end{equation*}
\end{lemma}
\begin{proof}
    We have 
    \begin{equation*}
    \begin{split}
        |\Delta_\psi(x,y,\chi)-\Delta_\vartheta(x,y,\chi)|\leq &|\Delta_\psi(x,y,1)-\Delta_\vartheta(x,y,1)|\\
        \ll& \log x\sum_{2\leq k\leq \log x} |(x+y)^{1/k}- x^{1/k}|\ll \sqrt{y}\log x.
    \end{split}
    \end{equation*}
    The cases $\Delta_\vartheta(x,y,q,a)$ and $\Delta_\vartheta(x,y,a/q)$ can be shown analogously.
\end{proof}
The next Lemma establishes a relationship between the weighted and unweighted prime counting functions $\vartheta$ and $\pi$.
\begin{lemma}
    Let $(a,q)=1$, let $\chi$ be a Dirichlet character modulo $q$ and let $\varepsilon>0$. For $0<y\leq x-x^\varepsilon$ we have
    \begin{equation}\label{eq: weighted to unweighted}
    \begin{split}
        \Delta_\pi(x,y,\chi)\ll& \frac{1}{\log x}\max_{0<t\leq y}|\Delta_\vartheta(x-y+t,t,\chi)|\\
        \Delta_\pi(x,y,q,a)\ll& \frac{1}{\log x}\max_{0<t\leq y}|\Delta_\vartheta(x-y+t,t,q,a)|\\
        \Delta_\pi(x,y,a/q)\ll& \frac{1}{\log x}\max_{0<t\leq y}|\Delta_\vartheta(x-y+t,t,a/q)|.
    \end{split}
    \end{equation}
    We may extend the range for the inequalities  \eqref{eq: weighted to unweighted} to $0<y\leq x$ by incurring an additional error term of $\cO(x^\varepsilon)$.
\end{lemma}
\begin{proof}
    We prove the Lemma for the case $\Delta_\pi(x,y,q,a)$. The proofs for the cases $\Delta_\pi(x,y,\chi)$ and $\Delta_\pi(x,y,a/q)$ work analogously. By Riemann-Stieltjes integration \mbox{(see \cite[Theorem A.2]{MV06})} we have
    \begin{equation}\label{eq: weighted to unweighted proof}
    \begin{split}
        \Delta_\pi(x,y,q,a)=&\int_{0}^y\frac{1}{\log (x-y+t)}\diff \bigg(\sum_{\substack{x-y<p\leq x-y+t\\p\equiv a\thinspace(q)}}\log p\bigg) -\frac{1}{\varphi(q)}\int_{0}^y\frac{1}{\log x-y+t}\diff t\\
        =&\int_{0}^y\frac{1}{\log (x-y+t)}\diff \Delta_\vartheta(x-y+t,t,q,a)\\
        =&\frac{\Delta_\vartheta(x,y,q,a)}{\log x}- \int_{0}^y\frac{\Delta_\vartheta(x-y+t,t,q,a)}{(x-y+t)\log^2(x-y+t)}\diff t .
    \end{split}
    \end{equation}
    For $0<y\leq x-x^\varepsilon$, \eqref{eq: weighted to unweighted} follows directly from \eqref{eq: weighted to unweighted proof}. For $x-x^\varepsilon<y\leq x$ we have
    \begin{equation*}
    \begin{split}
        \Delta_\pi(x,y,q,a)=&\frac{\Delta_\vartheta(x,y,q,a)}{\log x}- \int_{y-x+x^\varepsilon}^{y}\frac{\Delta_\vartheta(x-y+t,t,q,a)}{(x-y+t)\log^2(x-y+t)}\diff t\\
        &- \int_{0}^{y-x+x^\varepsilon}\frac{\Delta_\vartheta(x-y+t,t,q,a)}{(x-y+t)\log^2(x-y+t)}\diff t\\
        \ll&\frac{1}{\log x^\varepsilon}\max_{0<t\leq y}|\Delta_\vartheta(x-y+t,t,q,a)|+ \int_0^{x^\varepsilon}\frac{|\Delta_\vartheta(x-y+t,t,q,a)|}{(x-y+t)\log^2(x-y+t)}\diff t\\
        \ll&\frac{1}{\log x}\max_{0<t\leq y}|\Delta_\vartheta(x-y+t,t,q,a)|+ x^\varepsilon.\\
    \end{split}
    \end{equation*}
\end{proof}
\begin{corollary}
    Let $(a,q)=1$, let $\chi$ be a Dirichlet character modulo $q$ and let $\varepsilon>0$. For $0<y\leq x$ we have
    \begin{equation*}
    \begin{split}
        \Delta_\vartheta(x,y,\chi)\ll& \log x\max_{0<t\leq y}|\Delta_\pi(x-y+t,t,\chi)|\\
        \Delta_\vartheta(x,y,q,a)\ll& \log x\max_{0<t\leq y}|\Delta_\pi(x-y+t,t,q,a)|\\
        \Delta_\vartheta(x,y,a/q)\ll& \log x\max_{0<t\leq y}|\Delta_\pi(x-y+t,t,a/q)|.
    \end{split}
    \end{equation*}
\end{corollary}
We want to infer properties of prime counting functions with different twists from one another, instrumental therein is the well-known orthogonality relation
    \begin{equation}\label{eq: orthogonality relation dirichlet characters}
        \frac{1}{\varphi(q)}\sum_{\chi \thinspace(q)}\chi(b)\overline{\chi(a)}=\begin{cases}
            1 & \text{if }b\equiv a \thinspace(\text{mod } q)\\
            0 & \text{otherwise,}
        \end{cases}
    \end{equation}
where the sum runs over all Dirichlet characters $\chi$ modulo $q$ (see e.g. \cite[Equation 4.27]{MV06}).

\begin{lemma}\label{lemma: S as lin comb of psi}
    Let $(a,q)=1$, let $0<y\leq x$ and let $\kappa\in\{\pi,\vartheta,\psi\}$. We have
    \begin{equation}\label{eq: psi arith as lin comb of Dirichlet}
        \Delta_\kappa(x,y,q,a)=\frac{1}{\varphi(q)}\sum_{\chi\thinspace (q)}\overline{\chi(a)}\Delta_\kappa(x,y,\chi),
    \end{equation}
    and
    \begin{equation}\label{eq: S linear combination of chi}
        \Delta_\kappa(x,y,a/q)= \frac{1}{\varphi(q)} \sum_{\chi \thinspace(q)}\chi(a)\tau(\overline{\chi})\Delta_\kappa(x,y,\chi) + \mathcal{O}(\log q\log x),
    \end{equation}
    where $\tau(\chi)$ denotes the Gauss sum
    \begin{equation*}
        \tau(\chi):=\sum_{b \thinspace(q)}\chi(b)\e(b/q).
    \end{equation*}
\end{lemma}
\begin{proof}
    It suffices to show $\kappa=\psi$, the cases $\kappa\in\{\pi,\vartheta\}$ work analogously. Equation \eqref{eq: psi arith as lin comb of Dirichlet} follows immediately from Equation \eqref{eq: orthogonality relation dirichlet characters}. We simply multiply both sides by $\Lambda(n)-1$, sum over all $n\in(x-y,x]$, and exchange the order of summation. For Equation \eqref{eq: S linear combination of chi} we may omit the summands which are not coprime to $q$, thereby incurring a small error term.
    \begin{equation}\label{eq: coprime generating function}
        \Delta_\psi(x,y,a/q)= \sum_{\substack{x<n\leq x+y\\(n,q)=1}} \Lambda(n)\thinspace\e(na/q) -\frac{\mu(q)}{\varphi(q)}y + \mathcal{O}(\log q\log x).
    \end{equation}
    We may write $\e(na/q)$ as a sum over all additive characters modulo $q$ weighted by the indicator function \eqref{eq: orthogonality relation dirichlet characters}. We have
    \begin{equation}\label{eq: additive to Dirichlet}
        \e(an/q)=\frac{1}{\varphi(q)}\sum_{b \thinspace(q)} \e(b/q) \sum_{\chi \thinspace(q)}\chi(b)\overline{\chi(an)} = \frac{1}{\varphi(q)}\sum_{\chi \thinspace(q)}\overline{\chi(an)}\tau(\chi).
    \end{equation}
    We conclude the proof by inserting \eqref{eq: additive to Dirichlet} into \eqref{eq: coprime generating function}, exchanging the order of summation and noting that $\tau(\chi_0)=\mu(q)$ (see e.g. \cite[Theorem 4.1]{MV06}).
\end{proof}

Furthermore we are interested in the square error terms averaged over $(a,q)=1$.

\begin{lemma}\label{lemma: int psi to chi}
    Let $(a,q)=1$, let $0<y\leq x$ and let $\kappa\in\{\pi,\vartheta,\psi\}$. We have
    \begin{equation}\label{eq: int psi q a to chi}
    \begin{split}
        \psum_{a\thinspace (q)}|\Delta_\kappa(x,y,q,a)|^2  = \frac{1}{\varphi(q)}\sum_{\chi\thinspace(q)}|\Delta_\kappa(x,y,\chi)|^2 ,
    \end{split}
    \end{equation}
    and
    \begin{equation}\label{eq: int psi additive to chi}
    \begin{split}
        \psum_{a\thinspace (q)}|\Delta_\kappa(x,y,a/q)|^2  = \frac{1}{\varphi(q)}\sum_{\chi\thinspace (q)}|\tau(\chi)|^2 |\Delta_\kappa(x,y,\chi)|^2  +\mathcal{O}(q\log^2 q \log^2 x).
    \end{split}
    \end{equation}
\end{lemma}
\begin{proof}
    By Lemma \ref{lemma: S as lin comb of psi} we have
    \begin{equation*}
        \begin{split}
            \psum_{a\thinspace (q)} |\Delta_\kappa(x,y,a/q)|^2  =& \frac{1}{\varphi^{2}(q)}\psum_{a\thinspace (q)}\bigg| \sum_{\chi \thinspace(q)}\chi(a)\tau(\overline{\chi})\Delta_\kappa(x,y,\chi)\bigg|^2  +\mathcal{O}\bigg(q\log^2 q \log^2 x\bigg).
        \end{split}
    \end{equation*}
    By exchanging the order of integration and summation the main term becomes
    \begin{equation}\label{eq: int psi additive to chi main term}
        \begin{split}
            \frac{1}{\varphi^{2}(q)}\sum_{\chi_1,\chi_2 \thinspace(q)}\tau(\overline{\chi_1})\overline{\tau(\overline{\chi_2})}\Delta_\kappa(x,y,\chi_1)\overline{\Delta_\kappa(x,y,\chi_2)}\psum_{a\thinspace (q)}\chi_1\overline{\chi_2}(a).
        \end{split}
    \end{equation}
    We conclude the proof of \eqref{eq: int psi additive to chi} by noting that the sum over $a\thinspace (q)$ in \eqref{eq: int psi additive to chi main term} equals $\varphi(q)$ if $\chi_1=\chi_2$ and equals zero otherwise (see \cite[Corollary 4.5]{MV06}). The proof for \eqref{eq: int psi q a to chi} works analogously by omitting the Gauss sums $\tau$.
\end{proof}

\subsection{Zeroes of Dirichlet $L$-functions}
In the previous subsection we reduced various error terms of twisted prime counting functions in short intervals to functions of $\Delta_\psi(x,y,\chi)$. We may therefore infer information about these error terms via the truncated Riemann-von Mangoldt explicit formula.

\begin{theorem}\label{thm: chebyshev approx zeroes}
    Let $\chi$ be a Dirichlet character modulo $q$, let $0< y\leq x$ and let $2\leq T\leq x$. We have
    \begin{equation*}
        \Delta_\psi(x,y,\chi)=  -\sum_{\substack{\rho\thinspace(\chi)\\|\gamma|\leq T}} \frac{x^\rho -(x-y)^\rho}{\rho}  +\mathcal{O}\bigg( \frac{x}{T}\log^2 qx\bigg),
    \end{equation*}
    where the sum runs over all non-trivial zeroes $\rho=\beta+i\gamma$ of the Dirichlet $L$-function $L(s,\chi)$ with real part $\beta>0$ and imaginary part $|\gamma|\leq T$. 
\end{theorem}
\begin{proof}
    See \cite[Theorem 12.12]{MV06} and \cite[Equation 12.12]{MV06}.
\end{proof}
Notably, Cully-Hugill and Johnston \cite{CHJ25} have shown an improvement to the error term in Theorem \ref{thm: chebyshev approx zeroes} on average for the case $q=1$. This would lead to another saving of a logarithmic power for the main results about primes in all arithmetic progressions. Evidently by Theorem \ref{thm: chebyshev approx zeroes} we are able to bound the error terms of our various twisted prime counting functions, if we know the (approximate) locations of the zeroes of Dirichlet $L$-functions. Additionally to the partially proven Conditions \ref{hyp: zero-free region} \& \ref{hyp: zero-density}, we will be using the following two well-known results. The first result concerns the vertical distribution of zeroes in the critical strip.
\begin{theorem}\label{thm: vertical zeroes}
    Let $T\geq 4$. We have
    \begin{equation}
        \cN(0,T,\chi) = \frac{T}{\pi}\log \frac{qT}{2\pi} -\frac{T}{\pi}+\mathcal{O}(\log qT).
    \end{equation}
\end{theorem}
\begin{proof}
    See \cite[Corollaries 14.3 \& 14.7]{MV06}, which we can generalize to non-primitive characters via \cite[10.20]{MV06}.
\end{proof}

The zero-free regions in condition \ref{hyp: zero-free region} notably account for the possible existence of a real "exceptional" zero $\beta_1$, which could significantly worsen our results. We can however still give an upper bound on these zeroes with Siegel's theorem.
\begin{theorem}\label{thm: Siegel zero bound}
    If $\beta_0$ is a real zero of $L(s,\chi)$, then for any ${\varepsilon_0}>0$ there exists an (ineffective) constant $c_0>0$ only depending on ${\varepsilon_0}$, such that $\beta_0<1-\frac{c_0}{q^{\varepsilon_0}}$.
\end{theorem}
\begin{proof}
    See \cite[Corollary 11.15]{MV06}.
\end{proof}

\section{Proofs of Main results}
\subsection{Primes in all short arithmetic progressions}
In the preliminaries we reduced the error terms of various twisted prime counting functions to functions of $\Delta_\psi(x,y,\chi)$. It therefore suffices to show the main results in terms of these error terms. We start by showing a generalization of Ingham's proof \cite{Ing32} given a zero-free region. Theorem \ref{thm: Ingham} then follows from the following Lemma.

\begin{lemma}
    Assume Condition \ref{hyp: zero-free region}. Let $q\leq Q$ and $y\in [qx \exp(-\omega(x))\log^{2+\varepsilon} qx, x]$. We have
    \begin{equation*}
        \frac{1}{\varphi(q)}\sum_{\chi\thinspace (q)}|\Delta_\psi(x,y,\chi)| \ll \cB(x,q) + x\log^2 qx \exp(-\omega(x)).
    \end{equation*}
    We may omit the term $\cB$ if we restrict $q\leq (\frac{\log x}{\omega(x)})^\ell$ for any $\ell>0$.
\end{lemma}
\begin{proof}
    Obviously we have $|\frac{x^\rho-(x-y)^\rho}{\rho}|\ll \frac{x^\beta}{|\gamma|}$ for any $\rho$ with $\beta\geq 0$ and $0<y\leq x$. By \mbox{Theorems \ref{thm: chebyshev approx zeroes} \& \ref{thm: vertical zeroes}} we obtain
    \begin{equation*}
    \begin{split}
        \frac{1}{\varphi(q)}\sum_{\chi\thinspace (q)}|\Delta_\psi(x,y,\chi)|\ll& \frac{1}{\varphi(q)}\sum_{\chi\thinspace(q)}\sum_{\substack{\rho\thinspace(\chi)\\|\gamma|\leq T}} \frac{x^{\beta}}{|\gamma|}+\frac{x}{T}\log^2 qx\\
        \ll& x^{1-\eta(T)}\sum_{\substack{1\leq n\leq T}} \frac{\log qn}{n}+\frac{x}{T}\log^2 qx\\
        \ll& \thinspace\cB(x,q) + x\log^2 qx \exp(-\omega(x)),
    \end{split}
    \end{equation*}
    where $T$ was chosen as in $\omega(x)$. To omit the term $\cB$ we apply Theorem \ref{thm: Siegel zero bound} and obtain
    \begin{equation}\label{eq: omit siegel zeroes}
        \cB(x,q)\ll \frac{x}{\varphi(q)}\exp(-c_0 \frac{\log x}{q^{\varepsilon_0}})\ll \frac{x}{\varphi(q)}\exp(-c_0 \log^{1-\ell\varepsilon_0}x \thinspace\omega^{\ell\varepsilon_0}(x))\ll x\exp(-\omega(x)),
    \end{equation}
    for any $\varepsilon_0<1/\ell$.
\end{proof}

To prove Theorem \ref{thm: primes in all intervals} we begin with a simplification of the truncated Riemann-von Mangoldt explicit formula in Theorem \ref{thm: chebyshev approx zeroes}.
\begin{lemma}\label{lemma: all ints as sum of zeroes}
    Let $\chi$ be a Dirichlet character modulo $q$, let $0<y\leq x$ and $2\leq T\leq x$. We have
    \begin{equation*}
        \Delta_\psi(x,y,\chi)\ll \frac{y}{x}\sum_{\substack{\rho\thinspace(\chi)\\|\gamma|\leq T}} x^{\beta}+\frac{x}{T}\log^2 qx.
    \end{equation*}
\end{lemma}
\begin{proof}
    By Theorem \ref{thm: chebyshev approx zeroes} we have
\begin{equation*}
    |\Delta_\psi(x,y,\chi)|\leq \sum_{\substack{\rho\thinspace(\chi)\\|\gamma|\leq T}}\bigg|\frac{x^\rho-(x-y)^\rho}{\rho}\bigg|+\cO\bigg( \frac{x}{T}\log^2 qx\bigg).
\end{equation*}
    For $\beta\geq 0$ we have
\begin{equation*}
\begin{split}
    \bigg|\frac{x^\rho-(x-y)^\rho}{\rho}\bigg|=&\bigg|\int_{x-y}^{x} u^{\rho-1}\diff u\bigg|\leq \int_{x-y}^{x} u^{\beta-1}\diff u\ll y x^{\beta-1}.
\end{split}
\end{equation*}
\end{proof}

\begin{lemma}
    Assume Conditions \ref{hyp: zero-free region} \& \ref{hyp: zero-density} with $\eta_0<\frac{1}{A}$. Let $q\leq Q$ and 
    \begin{equation*}
        y\in\bigg[q^2x^{1-\frac{1}{A}}\exp\bigg(\frac{\eta^{-1}(x^{1/A})}{A}\log g(q,x)\bigg)\log^{1+\tau+\varepsilon} qx,x\bigg],
    \end{equation*} for any $\varepsilon>0$ and where $\tau=\tau(x,A,\eta)=\frac{1}{1+A\eta(x^{1/A})}$. We have
    \begin{equation*}
        \frac{1}{\varphi(q)}\sum_{\chi\thinspace(q)}|\Delta_\psi(x,y,\chi)| \ll \frac{y}{x}\cB(x,q)+ y \bigg(\frac{x^{1-\frac{1}{A}}}{y}\bigg)^{1-\tau}q^{1-2\tau}g^\tau(q,x)\log^{2-\tau} qx .
    \end{equation*}
    We may omit the term $\cB$ if we additionally restrict $q\leq \log^\ell x$ for any $\ell>0$.
\end{lemma}
\begin{proof}
By Riemann-Stieltjes integration (see \cite[Theorem A.2]{MV06}) we have
    \begin{equation}\label{eq: RS integration all intervals}
    \begin{split}
        \frac{1}{\varphi(q)}\sum_{\chi\thinspace(q)}\sum_{\substack{\rho\thinspace(\chi)\\|\gamma|\leq T}} x^{\beta-1} \ll & \frac{\cB(x,q)}{x}+\frac{1}{\varphi(q)x^{1/A}}\sum_{\chi\thinspace(q)}\cN(0,T,\chi)\\
        &+\frac{\log x}{\varphi(q)}\int_{1-1/A}^{1-\eta(T)}x^{\sigma-1}   \sum_{\chi\thinspace(q)} \cN(\sigma,T,\chi)  \diff \sigma\\
        \ll & \frac{\cB(x,q)}{x}+\frac{qT}{\varphi(q)x^{1/A}}\log qT\\
        &+\frac{g(q,T)\log x}{\varphi(q)}\bigg[\bigg(\frac{qT}{x^{1/A}}\bigg)^{A\eta(T)}-\frac{qT}{x^{1/A}} \bigg].
    \end{split}
    \end{equation}
    We take Lemma \ref{lemma: all ints as sum of zeroes}, average over $\chi\thinspace (q)$, insert \eqref{eq: RS integration all intervals} and choose
    \begin{equation*}
        T=\bigg(\frac{q^{1-A\eta(x^{1/A})}x^{1+\eta(x^{1/A})}}{y} \frac{\log qx}{g(q,x)}\bigg)^\frac{1}{1+A\eta(x^{1/A})}= o\bigg( \frac{x^\frac{1}{A}}{q\log x}\bigg).
    \end{equation*} 
    The possible omission of the term $\cB$ follows from \eqref{eq: omit siegel zeroes}.
\end{proof}

\subsection{Primes in almost all short arithmetic progressions}

As we have already found in the preliminaries, the variations on the Chebyshev functions are closely related. With Lemmas \ref{lemma: psi to vartheta}--\ref{lemma: int psi to chi} we reduced Theorems \ref{thm: Ingham 2} \& \ref{thm: primes in almost all intervals} to the integral $\int_X^{2X} |\Delta_\psi(u,h,\chi)|^2 \diff u$ averaged over $\chi\thinspace(q)$. Instead of evaluating the integral for an interval $h(X)$ dependent on $X$, it is easier to evaluate error terms of the form $\int_1^X |\Delta_\psi(u,\theta u,\chi)|\diff u$ for some $0<\theta\leq 1$. In the following lemma we make an appropriate substitution.

\begin{lemma}\label{lemma: substitution}
    Let $X\geq 1$ and let $0< h\leq X$. We have
    \begin{equation*}
        \int_X^{2X}|\Delta_\psi(u,h,\chi)|^2\diff u\ll \frac{X}{h} \int_{h/3X}^{3h/X}\int_X^{3X}|\Delta_\psi(u,\theta u,\chi)|^2\diff u \diff \theta.
    \end{equation*}
\end{lemma}
\begin{proof}
    For any constant $0<c\leq 1$ we have
    \begin{equation*}
        \int_X^{2X}|\Delta_\psi(u,cX,\chi)|^2\diff u\ll \int_{\frac{c}{2}}^c\int_X^{2X}|\Delta_\psi(u,\theta u,\chi)|^2\diff u \diff \theta.
    \end{equation*}
    It therefore suffices to prove this lemma for $0<h\leq X/6$. Let $0<h\leq X/6$, let $2h\leq v \leq 3h$ and $X\leq u \leq 2X$, so that $h\leq v-h\leq 2h$ and $X\leq u+h\leq 3X$. Therefore we get
    \begin{equation*}
    \begin{split}
        |\Delta_\psi(u,h,\chi)|^2 = & |\Delta_\psi(u,v,\chi) -\Delta_\psi(u+h,v-h,\chi)|^2\\
        \ll & |\Delta_\psi(u,v,\chi)|^2 +|\Delta_\psi(u+h,v-h,\chi)|^2
    \end{split}
    \end{equation*}
    
    Integrating over $u\in [X,2X]$, we obtain
    \begin{equation*}
    \begin{split}
         h\int_X^{2X} |\Delta_\psi(u,h,\chi)|^2\diff u \ll &\int_X^{2X} \int_{2h}^{3h}|\Delta_\psi(u,v,\chi)|^2 +|\Delta_\psi(u+h,v-h,\chi)|^2 \diff v \diff u\\
        \ll &\int_X^{2X} \int_{2h}^{3h}|\Delta_\psi(u,v,\chi)|^2\diff v\diff u +\int_X^{3X} \int_{h}^{2h}|\Delta_\psi(u,v,\chi)|^2 \diff v \diff u.
    \end{split}
    \end{equation*}
    We substitute $w:=\theta u$ with $h\leq w\leq 3h$, so that
    \begin{equation*}
        \frac{h}{3X}\leq \frac{h}{u}\leq \theta\leq \frac{3h}{u}\leq \frac{3h}{X}\leq \frac{1}{2},
    \end{equation*}
    and subsequently
    \begin{equation*}
    \begin{split}
         h\int_X^{2X} |\Delta_\psi(u,h,\chi)|^2\diff u \ll &\int_X^{3X} \int_{h}^{3h}|\Delta_\psi(u,w,\chi)|^2\diff w\diff u \\
        \ll &X\int_X^{3X} \int_{h/3X}^{3h/X}|\Delta_\psi(u,\theta u,\chi)|^2\diff \theta\diff u .
    \end{split}
    \end{equation*}
    Since the integrand is continuous on $[X,3X]\times [h/3X,3h/X]$, except for a subset of measure zero, we may exchange the order of integration, concluding the proof.
\end{proof}

For the next step we want to find an asymptotic upper bound in terms of the non-trivial zeroes of $L$-functions. This is also where an ingenious trick from Saffari-Vaughan \cite{SV77} comes in. By integrating over a small set of intervals which contain $[1,X]$, we lose a constant factor, but improve the rate of absolute convergence.

\begin{lemma}\label{lemma: int theta chi as sum over zeroes}
    Let $4\leq T\leq X$, and let $0<\theta\leq 1$. We have
    \begin{equation*}
    \begin{split}
        \int_1^{X}| \Delta_\psi(u,\theta u,\chi)|^2 \diff u\ll \sum_{\substack{\rho \thinspace (\chi)\\ |\gamma|\leq T}}X^{1+2\beta}\min(\theta^2,|\gamma|^{-2}) \log q(|\gamma|+2)+\frac{X^3}{T^2}\log^4 qX. 
    \end{split}
    \end{equation*}
\end{lemma}
\begin{proof} 
    Since $\int_1^4|\Delta(u,\theta u,\chi)|^2\diff u\ll 1$, it suffices to consider the integral over $[4,X]$. By non-negativity we have
    \begin{equation}\label{eq: add integral}
    \begin{split}
        \int_{4}^{X}| \Delta_\psi(u,\theta u,\chi)|^2 \diff u=\int_1^2 \int_{4}^{X}| \Delta_\psi(u,\theta u,\chi)|^2 \diff u\diff v\leq \int_1^2 \int_{2v}^{Xv}| \Delta_\psi(u,\theta u,\chi)|^2 \diff u \diff v.
    \end{split}
    \end{equation}
    By Theorem \ref{thm: chebyshev approx zeroes} we have for $4\leq T\leq X$
    \begin{equation}\label{eq: Yv to sum over zeroes}
    \begin{split}
        \int_{2v}^{Xv}| \Delta_\psi(u,\theta u,\chi)|^2 \diff u=& \int_{2v}^{Xv}\bigg|\sum_{\substack{\rho \thinspace (\chi)\\|\gamma|\leq T}} \frac{1-(1-\theta )^\rho}{\rho} u^\rho  \bigg|^2\diff u+\mathcal{O}\bigg(\frac{X^3}{T^2}\log^4 qX\bigg).
    \end{split}
    \end{equation}
    By Theorems \ref{thm: vertical zeroes} \& \ref{thm: Siegel zero bound}, as well as taking e.g. the classical zero-free region \mbox{\cite[Theorem 11.3]{MV06}}, the integral on the right hand side converges absolutely. Therefore we may exchange the order of integration and summation to obtain
    \begin{equation}\label{eq: Yv calc}
    \begin{split}
        \int_{2v}^{Xv}\bigg|\sum_{\substack{\rho \thinspace (\chi)\\|\gamma|\leq T}} \frac{1-(1-\theta )^\rho}{\rho} u^\rho \bigg|^2\diff u =&\sum_{\substack{\rho_1,\rho_2\thinspace(\chi)\\|\gamma_1|,|\gamma_2|\leq T}} \frac{1-(1-\theta )^{\rho_1}}{\rho_1}  \frac{1-(1-\theta )^{\overline{\rho_2}} }{\overline{\rho_2}} \int_{2v}^{Xv} u^{\rho_1+\overline{\rho_2}} \diff u\\
        =&\sum_{\substack{\rho_1,\rho_2\thinspace(\chi)\\|\gamma_1|,|\gamma_2|\leq T}} \frac{1-(1-\theta )^{\rho_1} }{\rho_1}  \frac{1-(1-\theta )^{\overline{\rho_2}} }{\overline{\rho_2}}\\
        &\times\frac{X^{1+\rho_1+\overline{\rho_2}}-2^{1+\rho_1+\overline{\rho_2}}}{1+\rho_1+\overline{\rho_2}}\thinspace v^{1+\rho_1+\overline{\rho_2}}.
    \end{split}
    \end{equation}
    By Theorems \ref{thm: vertical zeroes} \& \ref{thm: Siegel zero bound}, as well as taking e.g. the classical zero-free region \mbox{\cite[Theorem 11.3]{MV06}}, the double sum converges absolutely and uniformly in $v\in[1,2]$. Therefore we have
    
    \begin{equation}
    \begin{split}
        \int_1^2\int_{2v}^{Xv}\bigg|\sum_{\substack{\rho \thinspace (\chi)\\|\gamma|\leq T}} \frac{1-(1-\theta )^\rho }{\rho} u^\rho \bigg|^2\diff u\diff v \ll \sum_{\substack{\rho_1,\rho_2\thinspace(\chi)\\|\gamma_1|,|\gamma_2|\leq T}} \frac{1-(1-\theta )^{\rho_1} }{\rho_1}  \frac{1-(1-\theta )^{\overline{\rho_2}} }{\overline{\rho_2}} \frac{X^{1+\rho_1+\overline{\rho_2}}}{(1+\rho_1+\overline{\rho_2})^2}.
    \end{split}
    \end{equation}
    Applying the inequality $|ab|\leq |a|^2+|b|^2$ permits us to rearrange the symmetrical double sum to obtain

    \begin{equation}\label{eq: rearranged rho2}
    \begin{split}
        \int_1^2\int_{2v}^{Xv}\bigg|\sum_{\substack{\rho \thinspace (\chi)\\|\gamma|\leq T}} \frac{1-(1-\theta )^\rho }{\rho} u^\rho \bigg|^2\diff u\diff v \ll \sum_{\substack{\rho_1\thinspace(\chi)\\|\gamma_1|\leq T}}X^{1+2\beta_1} \frac{|1-(1-\theta )^{\rho_1}|^2}{\gamma_1^2} \sum_{\substack{\rho_2\thinspace(\chi)\\|\gamma_2|\leq T}}  \frac{1}{(1+|\gamma_1-\gamma_2|)^2}.
    \end{split}
    \end{equation}
    By applying Theorem \ref{thm: vertical zeroes}, the inner sum of \eqref{eq: rearranged rho2} can be bounded by 
    \begin{equation*}
    \begin{split}
        \sum_{\substack{\rho_2\thinspace(\chi)\\|\gamma_2|\leq T}}  \frac{1}{(1+|\gamma_1-\gamma_2|)^2}\ll& \sum_{\substack{\gamma_1\leq n\leq T}}  \frac{\log qn}{(1+|\gamma_1-n|)^2}+\sum_{\substack{-T\leq n<\gamma_1}}  \frac{\log qn}{(1+|\gamma_1-n|)^2}\\
        \ll& \log q(|\gamma_1|+2)\sum_{\substack{1\leq n\leq 2T}}  \frac{\log n}{n^2}\ll \log q(|\gamma_1|+2).
    \end{split}
    \end{equation*}
    
    For $\gamma \leq \theta^{-1}$, we take the identity
    
    \begin{equation*}
    \begin{split}
        |1-(1-\theta )^{\rho} |^2=&1-2 (1-\theta )^{\beta} \cos(\gamma\log(1-\theta )) +(1-\theta)^{2\beta}\\
        =& (1-(1-\theta )^{\beta})^2 +\mathcal{O}(\theta^2\gamma^2),
    \end{split}
    \end{equation*}
    and otherwise we take the trivial upper bound. We obtain
    
    \begin{equation}\label{eq: rearranged rho2 with upper bound}
    \begin{split}
        &\int_1^2\int_{2v}^{Xv}\bigg|\sum_{\substack{\rho \thinspace (\chi)\\|\gamma|\leq T}} \frac{1-(1-\theta )^\rho}{\rho} u^\rho \bigg|^2\diff u\diff v \ll \sum_{\substack{\rho_1\thinspace(\chi)\\|\gamma_1|\leq T}}X^{1+2\beta_1}\min(\theta^2,|\gamma_1|^{-2}) \log q(|\gamma_1|+2). 
    \end{split}
    \end{equation}

    Combining \eqref{eq: add integral}, \eqref{eq: Yv to sum over zeroes} \& \eqref{eq: rearranged rho2 with upper bound} concludes the proof.
    
\end{proof}

Just like for primes in all short arithmetic progressions, it proves useful to get a baseline result without invoking any zero-density estimates. This is because some logarithmic factors are lost for $\frac{1}{A}=\eta_0>0$, which particularly matters for the GRH case. The following Lemma in conjunction with Lemmas \ref{lemma: psi to vartheta}--\ref{lemma: int psi to chi} \& \ref{lemma: substitution} proves Theorem \ref{thm: Ingham 2}.

\begin{lemma}\label{lemma: almost all ingham psi}
    Assume Condition \ref{hyp: zero-free region}. Let $q\leq Q$ and $\theta\in[q X^{-2\eta(X)}\log^{2+\varepsilon} qX,1]$ for any $\varepsilon>0$. We have
    \begin{equation*}
        \frac{1}{\varphi(q)}\sum_{\chi\thinspace(q)}\int_{1}^{X}| \Delta_\psi(u,\theta u,\chi)|^2 \diff u\ll \theta^2 X \cB(X^2,q) \log q+\theta X^{3-2\eta(X)} \log^2 \frac{2q}{\theta}.
    \end{equation*}
    We may omit the term $\cB$ if we additionally restrict $q\leq \eta^{-\ell}(X)$ for any $\ell>0$.
\end{lemma}
\begin{proof}
    We apply Lemma \ref{lemma: int theta chi as sum over zeroes} and then Theorem \ref{thm: vertical zeroes} to obtain 
    \begin{equation}
    \begin{split}
        \frac{1}{\varphi(q)}\sum_{\chi\thinspace(q)}\int_{1}^{X}| \Delta_\psi(u,\theta u,\chi)|^2 \diff u \ll& \theta^2 X \cB(X^2,q)\log q+\frac{X^3}{T^2}\log^4 qX\\
        &+\frac{X^{3-2\eta(T)}}{\varphi(q)}\sum_{\chi\thinspace(q)}\bigg[\sum_{\substack{\rho \thinspace (\chi)\\ |\gamma|\leq \theta^{-1}}}\theta^2 \log \frac{2q}{\theta}+ \sum_{\substack{\rho \thinspace (\chi)\\ \theta^{-1}<|\gamma|\leq T}}\frac{\log q\gamma}{\gamma^2}\bigg]\\
        \ll& \theta^2 X \cB(X^2,q)\log q +\theta X^{3-2\eta(T)} \log^2 \frac{2q}{\theta}+\frac{X^3}{T^2}\log^4 qX
    \end{split}
    \end{equation}
    Here we may choose $T=X$. To omit the term $\cB$ we invoke Theorem \ref{thm: Siegel zero bound} and proceed as in \eqref{eq: omit siegel zeroes}.
\end{proof}

For the case $\frac{1}{A}>\eta_0$ we can do significantly better. The following Lemma in conjunction with Lemmas \ref{lemma: psi to vartheta}--\ref{lemma: int psi to chi} \& \ref{lemma: substitution} proves Theorem \ref{thm: primes in almost all intervals}.

\begin{lemma}
    Assume Conditions \ref{hyp: zero-free region} \& \ref{hyp: zero-density} with $\frac{1}{A}>\eta_0$. Let $q\leq Q$ and
    \begin{equation*}
        \theta\in [qX^{-\frac{2}{A}}\exp(\eta^{-1}(X^\frac{2}{A})\log(g(q,X)\log^{2+\varepsilon} qX),1]
    \end{equation*}
    for any $\varepsilon>0$. We have 
    \begin{equation*}
    \begin{split}
        \frac{1}{\varphi(q)}\sum_{\chi\thinspace(q)}\int_1^X |\Delta_\psi(u,\theta u,\chi)|^2\diff u \ll& \theta^2X\cB(X^2,q)\log q\\
        &+\theta^2X^3 \exp\bigg(\frac{\theta}{qX^{-\frac{2}{A}}}\bigg)^{-A\eta(X)}\frac{g(q,X)\log^2 qX}{\varphi(q)}  .
    \end{split}
    \end{equation*}
    We may omit the term $\cB$ if we additionally restrict $q\leq \eta^{\ell}(X)$ for any $\ell>0$.
\end{lemma}
\begin{proof}

    By Lemmas \ref{lemma: int theta chi as sum over zeroes} and Theorem \ref{thm: vertical zeroes} we have
    \begin{equation}\label{eq: int as sum over zeroes unconditional}
    \begin{split}
        \frac{1}{X}\int_{1}^{X}| \Delta_\psi(u,\theta u,\chi)|^2 \diff u \ll& \theta X^{2-\frac{2}{A}} \log^2 \frac{2q}{\theta}+ \theta^2 \log \frac{2q}{\theta}\sum_{\substack{\rho \thinspace (\chi)\\\beta> 1-\frac{1}{A}\\|\gamma|\leq \theta^{-1}}}X^{2\beta}\\
        &+\sum_{\substack{\rho \thinspace (\chi)\\\beta> 1-\frac{1}{A}\\\theta^{-1}< |\gamma|\leq T}}X^{2\beta}\frac{\log q\gamma}{\gamma^2}+\frac{X^2}{T^2}\log^4 qX.
    \end{split}
    \end{equation}
    We desire to evaluate \eqref{eq: int as sum over zeroes unconditional} averaged over all $\chi\thinspace(q)$. For this we look at each sum on the right hand side separately. By Riemann-Stieltjes integration (see e.g. \cite[Theorem A.2]{MV06}) we obtain
    
    \begin{equation}\label{eq: int sum smaller than theta-2}
    \begin{split}
        \frac{1}{\varphi(q)}\sum_{\chi\thinspace (q)}\sum_{\substack{\rho \thinspace (\chi)\\\beta> 1-\frac{1}{A} \\ |\gamma|\leq \theta^{-1}}}X^{2\beta} \ll& \cB(X^2,q)+  \frac{X^{2-\frac{2}{A}}}{\varphi(q)} \sum_{\chi\thinspace(q)} \cN(0,\theta^{-1},\chi)\\
        &+ \frac{\log X}{\varphi(q)}\int_{1-\frac{1}{A}}^{1-\eta(\theta^{-1})} X^{2u} \sum_{\chi\thinspace(q)}\cN(u,\theta^{-1},\chi)du\\
        \ll& \cB(X^2,q)+\frac{X^{2-\frac{2}{A}}}{\theta} \log \frac{2q}{\theta} \\
        &+\frac{g(q,\theta^{-1})}{\varphi(q)}\log X\int_{1-\frac{1}{A}}^{1-\eta(\theta^{-1})}  X^{2u}\bigg(\frac{q}{\theta}\bigg)^{A(1-u)} \diff u  \\
        \ll& \cB(X^2,q)+X^2\thinspace \bigg( \frac{\theta X^\frac{2}{A}}{q}\bigg)^{-A\eta(\theta^{-1})}\frac{g(q,\theta^{-1})}{\varphi(q)}\log X.
    \end{split}
    \end{equation}
    Analogously, we have 
    \begin{equation}\label{eq: int sum bigger than theta-2}
    \begin{split}
        \frac{1}{\varphi(q)}\sum_{\chi\thinspace (q)}\sum_{\substack{\rho \thinspace (\chi)\\\beta>1-\frac{1}{A}\\ \theta^{-1}<|\gamma|\leq T}}X^{2\beta}\frac{\log q\gamma}{\gamma^2} \ll& \theta X^{2-\frac{2}{A}}\log^2 \frac{2q}{\theta}\\
        &+\frac{\log X}{\varphi(q)}  \int_{\theta^{-1}}^T\int_{1-\frac{1}{A}}^{1-\eta(t)}X^{2u}\sum_{\chi\thinspace (q)}\cN(u,t,\chi) \frac{\log qt}{t^3} \diff u\diff t\\ 
        \ll& X^2\frac{g(q,T)\log^2 qX}{\varphi(q)}  \int_{\theta^{-1}}^T \bigg(\frac{X^2}{(qt)^A}\bigg)^{-\eta(t)}\frac{\diff t}{t^3}\\ 
        \ll& \theta^2X^2 \bigg( \frac{\theta X^\frac{2}{A}}{q}\bigg)^{-A\eta(T)}\frac{g(q,T)\log^2 qX}{\varphi(q)}   .
    \end{split}
    \end{equation}
    We combine \eqref{eq: int as sum over zeroes unconditional}, \eqref{eq: int sum smaller than theta-2} \& \eqref{eq: int sum bigger than theta-2}, and take $T=X^\frac{2}{A}$. To omit the term $\cB$ we invoke Theorem ~\ref{thm: Siegel zero bound} and proceed as in \eqref{eq: omit siegel zeroes}.
\end{proof}

\section*{Acknowledgments.} 
The author thanks their supervisors Bryce Kerr and Igor E. Shparlinski for their valuable input. This research has been supported by a University International Postgraduate Award (UIPA) and a University of New South Wales School of Mathematics scholarship.

\end{document}